\documentclass[12pt,twoside,reqno]{amsart}
\usepackage{amsmath}
\usepackage{amsthm}
\usepackage{amssymb}
\usepackage{amsrefs}
\usepackage{mathrsfs}
\usepackage{mathtools}
\usepackage[usenames]{color}
\usepackage[utf8]{inputenc}
\usepackage{latexsym}
\usepackage{yhmath}
\usepackage{graphicx}
\usepackage{ifthen}
\usepackage{pgf,tikz}
\usepackage{mathrsfs}
\usepackage[colorlinks=true,citecolor=blue]{hyperref}
\usetikzlibrary{arrows}
\usepackage{setspace}
\newtheorem{thm}{}[section]
\newtheorem{theorem}[thm]{Theorem}
\newtheorem{cor}[thm]{Corollary}
\newtheorem{definition}[thm]{Definition}

\newtheorem{prop}[thm]{Proposition}

\theoremstyle{remark}

\numberwithin{equation}{section}
\allowdisplaybreaks

\newcommand{\Env}[2][]{% 
	\ifthenelse{ \equal{#1}{} }  
	{\ensuremath{#2_{\mathsf{c}}}}  
	{\ensuremath{#2_{\mathsf{c},#1}}}
}

\DeclareMathOperator{\supp}{supp}
\DeclareMathOperator{\sgn}{sign}

\newcommand{\N}{\mathbb{N}}
\newcommand{\R}{\mathbb{R}}
\newcommand{\C}{\mathbb{C}}
\newcommand{\X}{\mathbb{X}}

\def\ba{\begin{eqnarray*}}
	\def\ea{\end{eqnarray*}}

\def\bee{\begin{equation}}
	
	\def\ene{\end{equation}}

\newcommand{\x}{\mathbf{x}}
\newcommand{\one}{\mathbf{1}}
\begin{document}
	\setcounter{page}{1}
	
	\title[1-greedy-like bases]{Non-linear approximation by $1$-greedy bases}
	
	\author[P.M. Bern\'a]{Pablo M. Bern\'a}
	\address{Pablo M. Bern\'a\\
		Departamento de Métodos Cuantitativos, CUNEF Universidad\\ 
		Madrid\\
		28040 Spain} 
	\email{pablo.berna@cunef.edu}
	
		\author[D. Gonz\'alez]{David Gonz\'alez}
	\address{David Gonz\'alez\\
		Universidad Cardenal Herrera‐CEU, CEU Universities \\ 
		San Bartolomé 55, 46115 \\ 
		Alfara del Patriarca, Valencia, Spain}
	\email{dgonzalezmoro@gmail.com}


	
	%\begin{document}
	
	\begin{abstract} 
		The theory of greedy-like bases started in 1999 when S. V. Konyagin and V. N. Temlyakov introduced in \cite{KT} the famous Thresholding Greedy Algorithm. Since this year, different greedy-like bases appeared in the literature, as for instance: quasi-greedy, almost-greedy and greedy bases. The purpose of this paper is to introduce some new characterizations of 1-greedy bases. Concretely, given a basis $\mathcal B=(\x_n)_{n\in\N}$ in a Banach space $\X$, we know that $\mathcal B$ is $C$-greedy with $C>0$ if $\Vert f-\mathcal G_m(f)\Vert\leq C\sigma_m(f)$ for every $f\in\X$ and every $m\in\N$, where $\sigma_m(f)$ is the best $m$th error in the approximation for $f$, that is, $\sigma_m(f)=\inf_{y\in\X : \vert \supp(y)\vert\leq m}\Vert f-y\Vert$. Here, we focus our attention when $C=1$ showing that a basis is 1-greedy if and only if $\Vert f-\mathcal G_1(f)\Vert=\sigma_1(f)$ for every $f\in\X$.
	\end{abstract}
	
	\maketitle
	
	\section{Introduction and background}
	Let $\X$ be a Banach space over the field $\mathbb F=\R$ or $\C$, and let $\mathcal B=(\x_n)_{n=1}^\infty$ be a semi-normalized Markushevich basis of $\X$, that is,
	\begin{itemize}
		\item $\overline{\text{span}(\x_n : n\in\N)}=\mathbb X$;
		\item there exists a unique sequence $(\x_n^*)_{n=1}^\infty\subset\X^*$ (called biorthogonal functionals) such that $\x_n^*(\x_j)=\delta_{n,j}$;
		\item $\overline{\text{span}(\x_n^*: n\in\mathbb N)}^{w^*}=\mathbb X^*$;
		\item there exist $c_1,c_2$ such that
		$$0<c_1\leq\inf_n\lbrace\Vert\x_n\Vert,\Vert \x_n^*\Vert\rbrace\leq\sup_n \lbrace \Vert\x_n\Vert,\Vert \x_n^*\Vert\rbrace\leq c_2<\infty.$$
	\end{itemize}
	Under these conditions, for every $f\in\mathbb X$, we have the series expansion
	$$f\sim \sum_{n=1}^{\infty}\x_n^*(f)\x_n,$$
	where $(\x_n^*(f))_{n\in\mathbb N}\in c_0$. We denote by $\X_f$ the subspace of $\X$ where the support of $f$ is finite, that is $\vert\text{supp}(f)\vert<\infty$ with $\text{supp}(f)=\lbrace n\in\N : \x_n^*(f)\neq 0\rbrace$. Moreover, we will use the quantity $\Vert(\x_n^*(f))_n\Vert_\infty=\sup_{n\in\supp(f)}\vert\x_n^*(f)\vert$.
	Also, related to a Banach space, we can define the so called \textit{indicator sums}: let $A$ be a finite set and $\varepsilon=(\varepsilon_n)_{n\in A}$ a collection of numbers of modulus one, that is, $\vert\varepsilon_n\vert=1$ for all $n\in A$ ($\vert\varepsilon\vert=1$ for short). We define the (signed) indicator sums as follows:
	$$\one_{A}[\mathcal B,\X]=\one_{A}:=\sum_{j\in A}\x_j,\; \one_{\varepsilon A}[\mathcal B,\X]=\one_{\varepsilon A}:=\sum_{j\in A}\varepsilon_j\x_j.$$
	Of course, if $\varepsilon\equiv 1$, $\one_{\varepsilon A}=\one_A$.
	
	In 1999, in \cite{KT}, S. V. Konyagin and V. N. Temlyakov introduced the Thresholding Greedy Algorithm (TGA): consider $\X$ a Banach space with a basis $\mathcal B=(\x_n)_{n\in\mathbb N}$ and take $f\in\mathbb X$ and $m\in\mathbb N$. We define the \textit{natural greedy ordering} of $f\in\X$ as a map $\rho:\N\rightarrow\N$ such that $\supp(f)\subseteq\rho(\N)$ and such that if $j<k$, then either $\vert\x_{\rho(j)}^*(f)\vert>\vert\x_{\rho(k)}^*(f)\vert$ or $\vert\x_{\rho(j)}^*(f)\vert=\vert\x_{\rho(k)}^*(f)\vert$ and $\rho(j)<\rho(k)$. Then, the \textit{$m$th greedy sum of $f$} is 
	$$\mathcal G_m[\mathcal B,\X](f)=\mathcal G_m(f):=\sum_{j=1}^{m}\x_{\rho(j)}^*(f)\x_{\rho(j)}.$$
	Alternatively, denoting by $A_m(f)=\lbrace \rho(1),\cdots,\rho(m)\rbrace\subset\supp(f)$, we can identify the $m$th greedy sum like the projection
	$$\mathcal G_m(f)=\sum_{n\in A_m(f)}\x_n^*(f)\x_n,$$
	where $A_m(f)$ is called the \textit{$m$th greedy set of $f$} and verifies the condition
	$$\min_{n\in A_m(f)}\vert \x_n^*(f)\vert\geq\max_{n\not\in A_m(f)}\vert x_n^*(f).$$
	
	Once we have the algorithm, the first natural question is when the algorithm converges. For that, S. V. Konyagin and V. N. Temlyakov introduced the notion of quasi-greediness: we say that $\mathcal B$ in a Banach space $\X$ is \textbf{quasi-greedy} if there is $C>0$ such that
	$$\Vert f-\mathcal G_m(f)\Vert\leq C\Vert f\Vert,\; \forall m\in\N, \forall f\in\X.$$
	The relation between this property and the convergence was given by P. Wojtaszczyk in \cite{W}, where he proved the following: a basis is quasi-greedy if and only if
	$$\lim_{m\rightarrow+\infty}\Vert f-\mathcal G_m(f)\Vert=0,\; \forall f\in\X.$$
	Another situation to study is when the algorithm produces, up to some constant, the best approximation. For that, we need to introduce the best $m$th error in the approximation: given $m\in\mathbb N$ and $f\in\X$,
	$$\sigma_m[\mathcal B,\X](f)=\sigma_m(f):=\left\lbrace\left\Vert f-\sum_{j\in B}a_j\x_j\right\Vert : \vert B\vert\leq m, a_j\in\mathbb F\right\rbrace.$$
	Hence, we want to know when $\Vert f-\mathcal G_m(f)\Vert$ is comparable to $\sigma_m(f)$ and for that, we have the concept of greedy bases.
	\begin{definition}[\cite{KT}]
		We say that a basis $\mathcal B$ in a Banach space $\X$ is \textbf{greedy} if there is a positive constant $C>0$ such that
		\begin{eqnarray}\label{defgreedy}
			\Vert f-\mathcal G_m(f)\Vert\leq C\sigma_m(f),\; \forall m\in\N, \forall f\in\X.
		\end{eqnarray}
		The least constant verifying \eqref{defgreedy} is denoted by $C_g[\mathcal B,\mathbb X]=C_g$ and we say that $\mathcal B$ is $C_g$-greedy.
	\end{definition}
	There are several examples of these type of bases. Some of them are the following:
	\begin{itemize}
		\item Every orthonormal basis in a Hilbert space $\mathbb H$ is $1$-greedy, that is,
		$$\Vert f-\mathcal G_m(f)\Vert=\sigma_m(f),\; \forall m\in\mathbb N, \forall f\in\mathbb H.$$
		\item The Haar system in $L_p(0,1)$, $1<p<\infty$, is $C_g$-greedy with
		$$C_g\approx \max\lbrace p,q\rbrace,$$
		where $q$ is the conjugate of $p$ (\cite{AAB}).
		\item The canonical basis in the space $\ell_p$, $1\leq p<\infty$, is $1$-greedy.
	\end{itemize}
	In \cite{KT}, the authors provide a nice characterization of theses bases using two properties: democracy and unconditionality.
	
	\begin{definition}
		We say that a basis $\mathcal B$ in a Banach space $\X$ is \textbf{suppression unconditional} if there is $K\geq 1$ such that
		\begin{eqnarray}\label{defuncon}
			\Vert f-P_A(f)\Vert\leq K\Vert f\Vert,\, \forall \vert A\vert<\infty, \forall f\in\X,
		\end{eqnarray}
		where $P_A$ is the projection operator, that is, $P_A(f)=\sum_{j\in A}\x_j^*(f)\x_j$. The least constant verifying \eqref{defuncon} is denoted by $K_s[\mathcal B,\X]=K_s$ and we say that $\mathcal B$ is $K_s$-suppression unconditional.
	\end{definition}
	
	\begin{definition}
		We say that a basis $\mathcal B$ in a Banach space $\X$ is \textbf{super-democratic} if there is a positive constant $C$ such that
		\begin{eqnarray}\label{defsup}
			\Vert \one_{\varepsilon A}\Vert\leq \Vert\one_{\eta B}\Vert,\; \forall \vert A\vert\leq\vert B\vert<\infty, \forall \vert\varepsilon\vert=\vert\eta\vert=1.
		\end{eqnarray}
		The least constant verifying \eqref{defsup} is denoted by $\Delta_s[\mathcal B,\X]=\Delta_s$ and we say that $\mathcal B$ is $\Delta_s$-super-democratic. In addition, if \eqref{defsup} is satisfied for $\varepsilon\equiv\eta\equiv 1$, we say that $\mathcal B$ is $\Delta_d$-democratic.
	\end{definition}
	
	With these definitions, the characterization proved by S. V. Konyagin and V. N. Temlyakov is the following one.
	\begin{theorem}[\cite{KT}]
		A basis $\mathcal B$ in a Banach space $\X$ is greedy if and only if the basis is suppression-unconditional and democratic. 
		Quantitatively, we have
		$$\max\lbrace K_s, \Delta_d\rbrace\leq C_g\leq K_s(1+\Delta_d).$$
	\end{theorem}
	The same characterization works replacing democracy by super-democracy obtaining the estimates
	\begin{eqnarray}\label{estim}
		\max\lbrace K_s, \Delta_s\rbrace\leq C_g\leq K_s(1+\Delta_s),
	\end{eqnarray}
	as we can see in \cite{BBG} or \cite{AABW}.

	Another characterization of greedy bases was given in \cite{BB} where the authors proved that is possible to substitute the error $\sigma_m(f)$ by the best $m$th error in the approximation by subspaces of dimension one.
	\begin{theorem}[\cite{BB}]
		A basis in a Banach space is greedy if and only if there is $C>0$ such that
		$$\Vert f-\mathcal G_m(f)\Vert \leq C\mathcal D_m(f),\; \forall m\in\mathbb N, \forall f\in\X,$$
		where
		$$\mathcal D_m[\mathcal B,\X](f)=\mathcal D_m(f):=\inf\left\lbrace\left\Vert f-\alpha \one_A\right\Vert: \vert A\vert\leq m, \alpha\in\mathbb F\right\rbrace.$$
	\end{theorem}
	Hence, we have different equivalences for greedy bases, but here we analyze the case of $1$-greedy bases. The constant was a case studied by F. Albiac and P. Wojtaszczyk in \cite{AW}, where they realized that there exists a basis that is $1$-suppression unconditional, $1$-(super)democratic and $2$-greedy. This example has two consequences:
	\begin{itemize}
		\item The first one is that the upper estimate \eqref{estim} is optimal.
		\item It is not possible to characterize $1$-greedy bases in terms of $1$-suppression unconditional and $1$-democratic bases. 
	\end{itemize}
	To get $1$-greediness from a theoretical point of view, they introduced the Property (A) that in \cite{DKOSZ} was renamed and extended to the notion of symmetry for largest coefficients.
	\begin{definition}
		We say that a basis $\mathcal B$ in a Banach space $\X$ is \textbf{symmetric for largest coefficients} if there is a positive constant $C$ such that
		\begin{eqnarray}\label{defsym}
			\Vert f+\one_{\varepsilon A}\Vert\leq C\Vert f+\one_{\eta B}\Vert,
		\end{eqnarray}
		for all $f\in\mathbb X$, $A, B, \varepsilon$ and $\eta$ such that $\Vert(\x_n^*(f))_n\Vert_\infty\leq 1$, $\vert A\vert\leq\vert B\vert<\infty$, $A\cap B=\emptyset$, $\text{supp}(f)\cap (A\cup B)=\emptyset$ and $\vert\varepsilon\vert=\vert\eta\vert=1$. The least constant verifying \eqref{defsym} is denoted by $\Delta[\mathcal B,\X]=\Delta$ and we say that $\mathcal B$ is $\Delta$-symmetric for largest coefficients.
	\end{definition}
	With that definition, the corresponding characterization of greedy bases can be found in \cite{DKOSZ}.
	
	\begin{theorem}[\cite{AW,DKOSZ}]\label{th2}
		A basis in a Banach space is greedy if and only if the basis is suppression unconditional and symmetric for largest coefficients. Quantitatively, 
		$$\max\lbrace K_s,\Delta\rbrace\leq C_g\leq K_s\Delta.$$
	\end{theorem}
	Then, with this new characterization, we can recover the case of $1$-greedy bases using bases that are $1$-suppression unconditional and $1$-symmetric for largest coefficients. The main goal of this paper is to analyze the main characterization of $1$-greedy bases in terms of the best $1$th error in the approximation and using the so called Property (Q$^*$).
	\begin{definition}[\cite{BB}]
		We say that a basis $\mathcal B$ in a Banach space $\X$ has the Property (Q$^*$) if there is $C>0$ such that
		\begin{eqnarray}\label{defq}
			\Vert f+\one_{\varepsilon A}\Vert\leq C\Vert f+y\Vert,
		\end{eqnarray}
		for any $f\in\X$, $y\in\mathbb X_f$, $A$ and $\varepsilon$ such that $\Vert (\x_n^*(f))\Vert_\infty\leq 1$, $\vert A\vert\leq \vert B\vert<\infty$ where the set $B$ is defined as $B=\lbrace n\in\text{supp}(y) : \vert\x_n^*(y)\vert=1\rbrace$, $\supp(f)\cap \supp(y)=\emptyset, \supp(f+y)\cap A=\emptyset$ and $\vert\varepsilon\vert=1$. The least constant verifying \eqref{defq} is denoted by $\mathcal Q[\mathcal B,\X]$ and we say that $\mathcal B$ has the Property (Q$^*$) with constant $\mathcal Q$.
	\end{definition}
	With that property, one of the results proved in \cite{BB} is the following one.
	
	\begin{theorem}\label{bb2}
		Let $\mathcal B$ be a basis in a Banach space. $\mathcal B$ is greedy if and only if the basis has the Property (Q$^*$). Quantitatively,
		$$\mathcal Q\leq C_g\leq \mathcal Q^2.$$
	\end{theorem}
	
	As we can observe, using the Property (Q$^*$) with constant $1$, it is possible to recover $1$-greediness. Here, the main result that we show is the following one.
	
	\begin{theorem}\label{main}
		Let $\mathcal B$ be a basis in a Banach space $\X$. The following are equivalent:
		\begin{itemize}
			\item[a)] The basis is $1$-greedy.
			\item[b)] For any $f\in\X$,
			\begin{eqnarray}\label{proof1}
				\Vert f-\mathcal G_1(f)\Vert= \sigma_1(f).
			\end{eqnarray}
			\item[c)] The basis is $1$-symmetric for largest coefficients and $1$-suppression unconditional.
			\item[d)] The basis has the Property (Q$^*$) with constant $\mathcal Q=1$.
		\end{itemize}
	\end{theorem}
	To show the proof, we will establish new characterizations of $1$-unconditionality and the Property (Q$^*$) with $\mathcal Q=1$. The structure of the paper is as follows: in Section \ref{section2}, we give new and known equivalences of $1$-suppression unconditional bases and the so called symmetry for largest coefficients. In Section \ref{section3}, we talk about one equivalence of the Property (Q$^*$) with constant 1 based on sets of one element. Finally, in Section \ref{section4} we show the proof of Theorem \ref{main} and Section \ref{section5} contains a remark about one consequence of Theorem \ref{main}.

	\section{Technical results}\label{section2}
	First of all, we study one characterization of $1$-suppression unconditional bases.
	
	\begin{prop}\label{1un}
		Let $\mathcal B$ be a basis in a Banach space $\X$. Then, $\mathcal B$ is $1$-suppression unconditional if and only if 
		\begin{eqnarray}\label{1u}
			\Vert f-\x_n^*(f)\x_n\Vert \leq \Vert f\Vert,\; \forall f\in\mathbb X, \forall n\in \text{supp}(f).
		\end{eqnarray}
	\end{prop}
	\begin{proof}
		Of course, if $\mathcal B$ is $1$-suppression unconditional, then we have \eqref{1u}. Assume now \eqref{1u} and take $f\in\mathbb X$ with finite support, $A\subseteq\text{supp}(f)$ finite with $A=\lbrace n_1,n_2,\dots,n_m\rbrace$. Then, if we define $g=f-\x_{n_1}^*(f)\x_{n_1}$,
		\begin{eqnarray}
			\nonumber	\left\Vert f-\sum_{j=1}^{2}\x_{n_j}^*(f)\x_{n_j}\right\Vert &=& \left\Vert g-\x_{n_2}^*(g)\right\Vert\underset{\eqref{1u}}{\leq}\Vert g\Vert\\
			\nonumber	&=&\left\Vert f-\x_{n_1}^*(f)\x_{n_1}\right\Vert\underset{\eqref{1u}}{\leq}\Vert f\Vert.
		\end{eqnarray}
		Assume now that
		\begin{eqnarray}\label{2u}
			\left\Vert f-\sum_{j=1}^{m-1}\x_{n_j}^*(f)\x_{n_j}\right\Vert\leq \Vert f\Vert.
		\end{eqnarray}
		Then, if we take $h= f-\sum_{j=1}^{m-1}\x_{n_j}^*(f)\x_{n_j}$,
		\begin{eqnarray}
			\nonumber	\left\Vert f-\sum_{j=1}^{m}\x_{n_j}^*(f)\x_{n_j}\right\Vert &=& \left\Vert h-\x_{n_m}^*(h)\right\Vert\underset{\eqref{1u}}{\leq}\Vert h\Vert\\
			\nonumber	&=&\left\Vert f-\sum_{j=1}^{m-1}\x_{n_j}^*(f)\x_{n_j}\right\Vert\underset{\eqref{2u}}{\leq}\Vert f\Vert.
		\end{eqnarray}
		Hence, the basis is $1$-suppression-unconditional for elements with finite support. Now, applying density, we can show that basis is $1$-suppression-unconditional: take $f\in\mathbb X$ and $g\in\mathbb X$ with finite support such that $\Vert f-g\Vert<\varepsilon$ for $\varepsilon>0$. Hence,
		\begin{eqnarray*}
			\Vert f-P_A(f)\Vert&=&\Vert f-g+g-P_A(g)+P_A(g)-P_A(f)\Vert\\
			&\leq&\Vert f-g\Vert+\Vert g-P_A(g)\Vert+\Vert P_A(f-g)\Vert\\
			&\leq& (1+\Vert P_A\Vert)\Vert f-g\Vert+\Vert g\Vert\\
			&\leq& (2+\Vert P_A\Vert)\Vert f-g\Vert+\Vert f\Vert\\
			&<&(2+\Vert P_A\Vert)\varepsilon+\Vert f\Vert.
		\end{eqnarray*}
		Taking now $\varepsilon\rightarrow 0$, we obtain the result.
	\end{proof}
	The same idea works to prove the following result for quasi-greedy bases.
	\begin{prop}
		Let $\mathcal B$ be a basis in a Banach space $\mathbb X$. Then, $\mathcal B$ is $1$-quasi-greedy if and only if
		\begin{eqnarray}\label{qg1}
			\Vert f-\mathcal G_1(f)\Vert\leq \Vert f\Vert,\; \forall f\in\mathbb X,
		\end{eqnarray}
	\end{prop}
	\begin{proof}
		It is only necessary to show that \eqref{qg1} implies quasi-greediness. For that, taking $m\in\N$ and $A_m(f)$ a greedy set of cardinality $m$, if $f_1=f-\mathcal G_1(f)$,
		\begin{eqnarray*}
			\Vert f-\mathcal G_2(f)\Vert=\Vert f_1-\mathcal G_1(f_1)\Vert\underset{\eqref{qg1}}{\leq}\Vert f_1\Vert=\Vert f-\mathcal G_1(f)\Vert\underset{\eqref{qg1}}{\leq}\Vert f\Vert.
		\end{eqnarray*}
		Now, assume that
		\begin{eqnarray}\label{qg2}
			\Vert f-\mathcal G_{m-1}(f)\Vert\leq \Vert f\Vert.
		\end{eqnarray}
		Taking $f_2=f-\mathcal G_{m-1}(f)$,
		\begin{eqnarray*}
			\Vert f-\mathcal G_m(f)\Vert=\Vert f_2-\mathcal G_1(f_2)\Vert\underset{\eqref{qg1}}{\leq}\Vert f_2\Vert\underset{\eqref{qg2}}{\leq}\Vert f\Vert,
		\end{eqnarray*}
		so the basis is $1$-quasi-greedy.
	\end{proof}
	Also, respect to unconditionality, we can found the following result in \cite{BB}.
	\begin{prop}\label{bb}
		Let $\mathcal B$ be a $K_s$-suppression unconditional basis of a Banach space $\mathbb X$. Let $f\in\mathbb X$, $A\subseteq\text{supp}(f)$ and $\varepsilon_n=\dfrac{\x_n^*(f)}{\vert\x_n^*(f)\vert}$ for $n\in A$. Then,
		$$\left\Vert\sum_{j\in B}\x_j^*(f)\x_j+t\mathbf 1_{\varepsilon A}\right\Vert\leq K_s\Vert f\Vert,$$
		for each $B\subset \text{supp}(f)\setminus A$ and $t\leq \min\lbrace\vert\x_n^*(f)\vert : n\in A\rbrace$.
	\end{prop}
	
	\begin{cor}\label{cor1}
		Let $\mathcal B$ be a $1$-suppression unconditional basis in a Banach space $\X$. Then, if $\varepsilon_m=\dfrac{\x_m^*(f)}{\vert\x_m^*(f)\vert}$ with $\lbrace m\rbrace \in \text{supp}(f)$, we have
		$$\left\Vert\sum_{j\in B}\x_j^*(f)\x_j+t\varepsilon_m\x_{m}\right\Vert\leq \Vert f\Vert,$$
		for each $B\subset \text{supp}(f)\setminus \lbrace m\rbrace$ and $t\leq \vert\x_m^*(f)\vert$.
	\end{cor}
	\begin{proof}
		Just apply the last proposition with $A=\lbrace m\rbrace$ and $K_s=1$.
	\end{proof}

	Considering now the symmetry for largest coefficients, the following result is well known.
	
	\begin{prop}[\cite{AA}]\label{aa}
		A basis is $1$-symmetric for largest coefficients if and only if
		$$\Vert f+\varepsilon_n \x_n\Vert=\Vert f+\eta_j\x_j\Vert,$$
		for every $f\in\mathbb X$ with $\Vert (\x_n^*(f))_n\Vert_\infty\leq 1$, $\vert\varepsilon_n\vert=\vert\eta_j\vert=1$ with $n,j\not\in\text{supp}(f)$.
	\end{prop}
	\begin{cor}\label{corsym}
		Let $\mathcal B$ be a $1$-symmetric for largest coefficients basis in a Banach space $\X$. Then,
		$$\Vert f\Vert\leq \Vert f-P_{\lbrace n\rbrace}(f)+t\eta_k\x_k\Vert,$$
		for every $f\in\X$, $t\geq \Vert (\x_n^*(f))_n\Vert_\infty$, $n\in\text{supp}(f)$, $k\not\in\text{supp}(f)$ and $\vert\eta_k\vert=1$.
	\end{cor}
	\begin{proof}
		Take $f\in\X$ and $j,k\not\in\supp(f)$ and $\vert\varepsilon_j\vert=\vert\eta_k\vert=1$. Now, if we define $f'=\dfrac{f}{t}$ where $t\geq\Vert(\x_n^*(f))\Vert_\infty$, $\Vert(\x_n^*(f'))\Vert_\infty\leq 1$, so if we apply Proposition \ref{aa}, we obtain
		\begin{eqnarray}\label{corr}
			\Vert f'+\varepsilon_j\x_j\Vert=\Vert f'+\eta_k\x_k\Vert.
		\end{eqnarray}
		Multiplying now by $t$ in \eqref{corr}, we have
		\begin{eqnarray}\label{corr1}
			\Vert f+t\varepsilon_j\x_j\Vert=\Vert f+t\eta_k\x_k\Vert.
		\end{eqnarray}
		Consider now $n\in\supp(f)$ and define $g:=f-P_{\lbrace n\rbrace}(f)$. Applying $n$ as $j$ in \eqref{corr1} we obtain
		\begin{eqnarray}\label{last}
			\Vert g+t\varepsilon_n\x_n\Vert=\Vert g+t\eta_k\x_k\Vert,
		\end{eqnarray}
		where applying convexity,
		\begin{eqnarray*}
			\Vert f\Vert=\Vert f-P_{\lbrace n\rbrace}(f)+P_{\lbrace n\rbrace}(f)\Vert\leq\sup_{\vert\varepsilon_n\vert=1}\Vert f-P_{n}(f)+t\varepsilon_n\x_n\Vert\underset{\eqref{last}}{\leq}\Vert f-P_{\lbrace n\rbrace}(f)+t\eta_k\x_k\Vert.
		\end{eqnarray*}
	\end{proof}
	
	\section{Property (Q$^*$)}\label{section3}
	We analyze in this section an equivalence of the Property (Q$^*$) for the constant $1$. Remember that a basis has the Property (Q$^*$) with constant $\mathcal Q$ if
	$$\Vert f+\one_{\varepsilon A}\Vert\leq \mathcal Q\Vert f+y\Vert,$$
	for any $f\in\X$, $y\in\mathbb X_f$, $\Vert (\x_n^*(f))_n\Vert_\infty\leq 1$, $\vert A\vert\leq \vert B\vert<\infty$ with $B=\lbrace n\in\text{supp}(y) : \vert\x_n^*(y)\vert=1\rbrace$, $\supp(f)\cap \supp(y)=\emptyset, \supp(f+y)\cap A=\emptyset$ and $\vert\varepsilon\vert=1$.
	
	We note that thanks to \cite[Lemma 3.2]{BDKOW}, we can replace $f\in\mathbb X$ by $f\in\mathbb X_f$ thanks to a density argument in the last definition. 
	
	The main result here is the following equivalence.
	\begin{theorem}\label{1sym}
		Let $\mathcal B$ a basis in a Banach space $\X$. $\mathcal B$ has the Property (Q$^*$) with $\mathcal Q=1$ if and only if
		\begin{eqnarray}\label{refq}
			\Vert f+\varepsilon_n \x_n\Vert\leq \Vert f+\eta_k\x_k +y\Vert,
		\end{eqnarray}
		for any $f\in\X$, $y\in \X_f$, $\Vert (\x_n^*(f))_n\Vert_\infty\leq 1$, $n,k$ different indices such that $\text{supp}(f)\cap\text{supp}(y)=\emptyset$, $\text{supp}(f+y)\cap\lbrace n,k\rbrace=\emptyset$ and $\vert\varepsilon_n\vert=\vert\eta_k\vert=1$.
	\end{theorem}
	\begin{proof}
		Assume that we have the Property (Q$^*$) with $\mathcal Q=1$ and take now $f,\varepsilon_n, \eta_k$ and $y$ as in \eqref{refq}. Taking in \eqref{defq} $A=\lbrace n\rbrace$ and $y'=y+\eta_k\x_k$ with $B=\lbrace k\rbrace$,
		$$\Vert f+\varepsilon_n\x_n\Vert \underset{\eqref{defq}}{\leq} \Vert f+y'\Vert=\Vert f+\eta_k\x_k+y\Vert,$$
		so we obtain \eqref{refq}. 
		
		Assume now that we have \eqref{refq} and take $f,y,A,B,\varepsilon$ and $\eta$ as in the definition of the Property (Q$^*$), that is, as in \eqref{defq}. First of all, we do the following decomposition of $A$ and $B$: $A=\lbrace n_1,\dots,n_k\rbrace$ and $B=\lbrace m_1,\dots,m_p\rbrace$ with $p\geq k$ and $y=y_r+\sum_{j\in B}\eta_j\x_j$. Hence,
		\begin{eqnarray*}
			\Vert f+\varepsilon_{n_1}\x_{n_1}+\varepsilon_{n_2}\x_{n_2}\Vert &=&	\Vert (f+\varepsilon_{n_1}\x_{n_1})+\varepsilon_{n_2}\x_{n_2}\Vert\\
			&\underset{\eqref{refq}, y=0}{\leq}& \Vert (f+\varepsilon_{n_1}\x_{n_1})+\eta_{m_1}\x_{m_1}\Vert\\
			&=& \Vert (f+\eta_{m_1}\x_{m_1})+\varepsilon_{n_1}\x_{n_1}\Vert\\
			&\underset{\eqref{refq},y=0}{\leq}& \Vert f+\eta_{m_1}\x_{m_1}+\eta_{m_2}\x_{m_2}\Vert.
		\end{eqnarray*}
		Assume now that
		\begin{eqnarray}\label{refq2}
			\left\Vert f+\sum_{j=1}^{k-1}\varepsilon_{n_j}\x_{n_j}\right\Vert\leq \left\Vert f+\sum_{j=1}^{k-1}\eta_{m_j}\x_{m_j}\right\Vert.
		\end{eqnarray}
		Then,
		\begin{eqnarray*}
			\Vert f+\mathbf 1_{\varepsilon A}\Vert&=&\left\Vert (f+\varepsilon_{n_k}\x_{n_k})+\sum_{j=1}^{k-1}\varepsilon_{n_j}\x_{n_j}\right\Vert\\
			&\underset{\eqref{refq2}}{\leq}& \left\Vert (f+\varepsilon_{n_k}\x_{n_k})+\sum_{j=1}^{k-1}\eta_{m_j}\x_{m_j}\right\Vert\\
			&=& \left\Vert \left(f+\sum_{j=1}^{k-1}\eta_{m_j}\x_{m_j}\right)+\varepsilon_{n_k}\x_{n_k}\right\Vert\\
			&\underset{\eqref{refq},y_1=y_r+\sum_{j=k+1}^p\eta_{m_j}\x_{m_j}}{\leq}&\left\Vert\left(f+\sum_{j=1}^{k-1}\eta_{m_j}\x_{m_j}\right)+\eta_{m_k}\x_{n_k}+y_1\right\Vert\\
			&=&\left\Vert\left(f+\sum_{j=1}^{k-1}\eta_{m_j}\x_{m_j}\right)+\eta_k\x_k + y_r+\sum_{j=k+1}^p\eta_{m_j}\x_{m_j}\right\Vert\\
			&=&\Vert f+y\Vert.
		\end{eqnarray*}
	\end{proof}
	
	\section{Proof of Theorem \ref{main}}\label{section4}
	To study the proof of Theorem \ref{main}, we will use the following results proving the equivalence between the Property (Q$^*$) with $\mathcal Q=1$ with the $1$-symmetry of largest coefficients and $1$-suppression unconditionality.
	\begin{prop}\label{tech}
		Let $\mathcal B$ be a basis of a Banach space $\X$. The following are equivalent:
		\begin{itemize}
			\item[a)] $\mathcal B$ has the Property (Q$^*$) with $\mathcal Q=1$.
			\item[b)] $\mathcal B$ is $1$-suppression unconditional and $1$-symmetric for largest coefficients.
		\end{itemize}
	\end{prop}
	\begin{proof}
		Assume a). By Proposition \ref{1un}, we only have to show that
		$$\Vert f-\x_j^*(f)\x_j\Vert\leq\Vert f\Vert,\; \forall f\in\mathbb X,$$
		with $\lbrace j\rbrace\subset\text{supp}(f)$. Take then $f\in\X$ and $j\in\text{supp}(f)$ and define $f_1=\dfrac{f}{\Vert (\x_n^*(f))_n\Vert_\infty}$ and $g=f_1-\x_j^*(f_1)$. Hence, taking in \eqref{defq} $A=\emptyset$ and $y=\x_j^*\left(f_1\right)$,
		$$\Vert g\Vert=\Vert f_1-\x_j^*\left(f_1\right)\Vert\leq \Vert g+\x_j^*\left(f_1\right)\Vert=\Vert f_1\Vert.$$
		Then, we have that
		$$\left\Vert \dfrac{f}{\Vert (\x_n^*(f)_n\Vert_\infty}-\x_j^*\left(\dfrac{f}{\Vert (\x_n^*(f))_n\Vert_\infty}\right)\right\Vert\leq \left\Vert \dfrac{f}{\Vert (\x_n^*(f))_n\Vert_\infty}\right\Vert.$$
		Now, since $\x_j^*$ are linear for any $j$,
		$$\left\Vert \dfrac{f}{\Vert (\x_n^*(f))_n\Vert_\infty}-\dfrac{1}{\Vert (\x_n^*(f))_n\Vert_\infty}\x_j^*\left(f\right)\right\Vert\leq \left\Vert \dfrac{f}{\Vert (\x_n^*(f))_n\Vert_\infty}\right\Vert,$$
		and this implies that
		$$\Vert f-\x_j^*(f)\x_j\Vert\leq\Vert f\Vert,\; \forall f\in\mathbb X,$$
		so the basis is $1$-suppression unconditional.
		
		%	By Theorem \ref{1sym}, we know that
		%\begin{eqnarray*}
		%	\Vert f+\varepsilon_n \x_n\Vert\leq \Vert f+\eta_k\x_k +y\Vert,
		%\end{eqnarray*}
		%for any $f\in\X$, $y\in \X_f$, $\Vert (\x_n^*(f))_n\Vert_\infty\leq 1$, $n,k$ different indices such that $\text{supp}(f)\cap\text{supp}(y)=\emptyset$, $\text{supp}(f+y)\cap\lbrace n,k\rbrace=\emptyset$ and $\vert\varepsilon_n\vert=\vert\eta_k\vert=1$.

		Now, we need to show that the basis is $1$-symmetric for largest coefficients. For that, invoking Proposition \ref{aa}, so we only have to show that
		$$\Vert f+\varepsilon_n \x_n\Vert=\Vert f+\eta_j\x_j\Vert,$$
		for every $f\in\mathbb X$, $\vert\varepsilon_n\vert=\vert\eta_j\vert=1$ with $n,j\not\in\text{supp}(f)$, but this equality is trivial taking in \eqref{refq} $y=0$. Thus, a) is done.
		
		Prove now b). Assume that $\mathcal B$ is $1$-suppression unconditional and $1$-symmetric for largest coefficients. We have to show that
		\begin{eqnarray*}
			\Vert f+\varepsilon_n \x_n\Vert\leq \Vert f+\eta_k\x_k +y\Vert,
		\end{eqnarray*}
		for any $f\in\X,y\in \X_f$, $\Vert (\x_n^*(f))_n\Vert_\infty\leq 1$, $n,k$ different indices such that $\text{supp}(f)\cap\text{supp}(y)=\emptyset$, $\text{supp}(f+y)\cap\lbrace n,k\rbrace=\emptyset$ and $\vert\varepsilon_n\vert=\vert\eta_k\vert=1$.
		
		First of all, using the $1$-symmetry for largest coefficients (Proposition \ref{aa}), we have
		\begin{eqnarray*}
			\Vert f+\varepsilon_n \x_n\Vert\leq \Vert f+\eta_k\x_k\Vert.
		\end{eqnarray*}
		Now, if we define $f':= f+\eta_k\x_k+y$, if $A=\text{supp}(y)$, applying $1$-suppression unconditionality, we obtain
		$$\Vert f+\eta_k\x_k\Vert=\Vert f'-P_A(f')\Vert\leq \Vert f'\Vert=\Vert f+\eta_k\x_k+y\Vert,$$
		so the basis has the Property (Q$^*$) with $\mathcal Q=1$.
	\end{proof}
	\begin{prop}\label{tech2}
		Let $\mathcal B$ be a basis in a Banach space $\X$. If $\mathcal B$ is $1$-suppression unconditional and $1$-symmetric for largest coefficients, then
		$$\Vert f-\mathcal G_1(f)\Vert=\sigma_1(f),\; \forall f\in\X.$$
	\end{prop}
	\begin{proof}
		Take $f\in\X$, $n_g$ the index such that $P_{n_g}(f)=\mathcal G_1(f)$ and $y\in\X_f$ such that $\sigma_1(f)=\Vert f-y\Vert$ with $n_a=\supp(y)$. Taking now $\varepsilon\equiv\lbrace \sgn(\x_n^*(f))\rbrace$,
		\begin{eqnarray}\label{us1}
			\Vert f-\mathcal G_1(f)\Vert&\underset{\text{Corollary}\,  \ref{corsym}}{\leq}&\Vert f-P_{n_g}(f)-P_{n_a}(f)+t\varepsilon_{n_g}\x_{n_g}\Vert\\
			&=&\Vert P_{(n_g\cup n_a)^c}(f-y)+t\varepsilon_{n_g}\x_{n_g}\Vert,
		\end{eqnarray}
		where $t=\vert \x_{n_g}^*(f)\vert$. Now, since $(n_g\cup n_a)^c\subset \text{supp}(f-y)\setminus n_g$,
		\begin{eqnarray}\label{us2}
			\Vert P_{(n_g\cup n_a)^c}(f-y)+t\varepsilon_{n_g}\x_{n_g}\Vert\underset{\text{Corrolary}\, \ref{cor1}}{\leq}\Vert f-y\Vert.
		\end{eqnarray}
		Hence, by \eqref{us1} and \eqref{us2}, we obtain
		$$\Vert f-\mathcal G_1(f)\Vert\leq\sigma_1(f),$$
		and then, since $\sigma_1(f)\leq \Vert f-\mathcal G_1(f)\Vert,$
		we obtain that
		$$\Vert f-\mathcal G_1(f)\Vert=\sigma_1(f),\; \forall f\in\X.$$
		
		Now, assume that 
		$$\Vert f-\mathcal G_1(f)\Vert=\sigma_1(f),\; \forall f\in\X.$$
		We need to show that $\mathcal B$ is $1$-suppression unconditional and $1$-symmetric for largest coefficients. We start with the symmetry. Take $f,\varepsilon$ and $\eta$ as in the Proposition \ref{aa} and define the element $h:=f+\varepsilon_n\x_n+(1+\gamma)\eta_k\x_k$ with $\gamma>0$. Hence, the set $\lbrace k\rbrace $ is a greedy set with cardinality $1$. Then
		\begin{eqnarray*}
			\Vert f+\varepsilon_n\x_n\Vert=\Vert h-\mathcal G_1(h)\Vert\underset{\eqref{proof1}}{\leq}\sigma_1(h)\leq\Vert h-\varepsilon_n\x_n\Vert=\Vert f+(1+\gamma)\eta_k\x_k\Vert.
		\end{eqnarray*}
		Taking the limit when $\gamma\rightarrow 0$, the basis is $1$-symmetric for largest coefficients. Now, is the turn of $1$-suppression unconditionality. For that, take any $\lbrace j\rbrace\in\supp(f)$ with $f\in\mathbb X$ and $\alpha$ big enough such that $(\alpha+\x_j^*(f))$ is bigger than $\Vert(\x_n^*(f))_n\Vert_\infty$. Then, define now
		$$g:=(f-\x_j^*(f))+(\alpha+\x_j^*(f))\x_j.$$
		It is clearthat $\lbrace j\rbrace$ is now the $1$-greedy set of $g$. Hence,
		\begin{eqnarray*}
			\Vert f-\x_j^*(f)\Vert = \Vert g-\mathcal G_1(g)\Vert\underset{\eqref{proof1}}{\leq}\sigma_1(g)\leq\Vert g-\alpha\x_j\Vert=\Vert f\Vert,
		\end{eqnarray*}
		so the basis is $1$-suppression unconditional.
	\end{proof}
	
	\begin{proof}[Proof of Theorem \ref{main}]  Of course, a) implies b) since if $\mathcal B$ is $1$-greedy, then we have \eqref{proof1}, that is,
		\begin{eqnarray*}
			\Vert f-\mathcal G_1(f)\Vert=\sigma_1(f),\; \forall f\in\mathbb X,
		\end{eqnarray*}
		Now b) is equivalent to c) using Proposition \ref{tech2}. Also, using now Proposition \ref{tech}, c) is equivalent to d). Finally, we have to show that d) implies a), but this is trivial invoking Theorem \ref{bb2}.
	\end{proof}
	
	\section{Another consequences of Theorem \ref{main}}\label{section5}
	Some years ago, T. Oikhberg introduced in \cite{O} a variant of the Thresholding Greedy Algorithm for gaps, that is, selecting a sequence of positive integers $\mathbf n=(n_i)_{i\in\N}$ with $n_1<n_2<\dots$, the \textbf{$\mathbf n$-greedy bases} are as those bases where there exists a positive constant $C$ such that
	\begin{eqnarray}\label{ng}
		\Vert f-\mathcal G_n(f)\Vert\leq C\sigma_n(f),\; \forall n\in\mathbf n, \forall f\in\X.
	\end{eqnarray}
	Of course, if $\mathcal B$ is a greedy basis with constant $C_g$, then \eqref{ng} is satisfied with constant $C\leq C_g$. The most surprise result proved in \cite[Section 1]{O} is that if \eqref{ng} is satisfied with constant $C$, then $\mathcal B$ is greedy with constant $C_g\leq C$, that is, both notions are equivalent using the same constant!
	
	Hence, applying our Theorem \ref{main}, we have the following corollary.
	\begin{cor}
		Let $\mathcal B$ a basis in a Banach space $\X$. The following are equivalent:
		\begin{itemize}
			\item $\mathcal B$ is $1$-greedy.
			\item For every $f\in\X$,
			$$\Vert f-\mathcal G_1(f)\Vert=\sigma_1(f).$$
			\item The basis satisfied the condition \eqref{ng} with constant $C=1$, that is, the basis is $\mathbf n$-greedy with constant 1.
		\end{itemize}
	\end{cor}

	\bigskip
	\noindent \textbf{Funding :} P. M. Bern\'a was partially supported by the Grant PID2019-105599GB-I00 (Agencia Estatal de Investigación, Spain). D. González was partially supported by  ESI International Chair@ CEU-UCH.

	%\section{Proof of Lemma~\ref{rate3}}\label{appendix:a2}
	%
	%	\begin{thebibliography}{99}
		%	\bibitem{FalcoHackbusch}\textsc{A.~Falc\'o and W.~Hackbusch}, \textit{On minimal subspaces in tensor representations}, {\em Preprint 70/2010} Max Planck Institute for Mathematics in the Sciences, 2010.
		%	
		%	\bibitem{FN}\textsc{A. Falc\'o, A. Nouy},\textit{Proper generalized decomposition for nonlinear convex problems in tensor Banach spaces}, Number. Math. (2012) 121:503-530.
		%	\bibitem{T1}\textsc{V. N. Temlyakov},\textit{Greedy approximations with regard to bases}, Proceedings of the International Congress of Mathematicians. European Mathematical Society, 2 edition, 2006.
		%\end{thebibliography}
		%

		%% ------------------------------------------------------------------------
	\end{document}